\documentclass{article}
\usepackage[utf8]{inputenc}

\usepackage{amssymb,amsmath,amsfonts,amsthm}
\usepackage{latexsym}

\usepackage{hyperref}

\usepackage{xcolor}

\usepackage{graphicx} 

\theoremstyle{plain}
\newtheorem{thm}{Theorem}
\newtheorem{prop}[thm]{Proposition}
\newtheorem{cor}[thm]{Corollary}

\newtheorem{defin}[thm]{Definition}
\newtheorem{ex}[thm]{Example}
\newtheorem{rmk}[thm]{Remark}

\newcommand*{\R}{\mathbb{R}}

\newcommand*{\y}{\mathbf{y}}
\newcommand*{\z}{\mathbf{z}}
\newcommand*{\uu}{\mathbf{u}}

\newcommand*{\w}{\mathbf{w}}
\renewcommand*{\i}{\mathbf{1}}

\title{A note on Steinerberger's curvature for graphs}
\author{David Cushing, Supanat Kamtue, Erin Law, \\Shiping Liu, Florentin M\"unch, Norbert Peyerimhoff}
\date{\today}

\begin{document}

\maketitle

\abstract{In this note, we provide Steinerberger curvature formulas for block graphs, discuss curvature relations between two graphs and the graph obtained by connecting them via a bridge, and show that self-centered Bonnet-Myers sharp graphs are precisely those which are antipodal. We also discuss similarities and differences between Steinerberger and Ollivier Ricci curvature results.}

\tableofcontents

\section{Introduction}

Stefan Steinerberger introduced a non-local curvature notion for the vertices of a finite combinatorial graph $G= (V,E)$ with vertex set $V$ and edge set $E$ in his 2022 article \cite{St-23}. The idea is to consider solutions $K$ of the equation
\begin{equation} \label{eq:K-defin}
D K = n \cdot \mathbf{1}_n, 
\end{equation}
where the vertices of $G$ are enumerated by
$$ V = \{v_1,v_2,\dots,v_n\}, $$
$D = (d(v_i,v_j))_{i,j=1}^n$ is the distance matrix of $G$, and $K$ and $\mathbf{1}_n$ are both column vectors of length $n$ with all entries of $\mathbf{1}_n$ being equal to $1$. If equation \eqref{eq:K-defin} has a unique solution $K \in \mathbb{R}^n$, the curvature of vertex $v_i$ is then the $i$-th entry $K_i$ of the vector $K$. If there are multiple vectors $K$ solving equation \eqref{eq:K-defin}, the vector $K$ defining the vertex curvatures is one whose smallest entry is largest, that is, one for which $\min_i K_i$ is maximal. Finally, if there are no solutions for equation \eqref{eq:K-defin},
Steinerberger defines this new curvature with the help of the Moore-Penrose pseudoinverse $D^\dagger$ of $D$ (which exists always) by setting
$$ K = n \cdot D^\dagger \mathbf{1}_n. $$
Interestingly, the ``total curvature'', given by 
$$ \Vert K \Vert_{\ell^1} = \sum_i K_i $$
for any solution $K \in \mathbb{R}_{\ge 0}^n$ of \eqref{eq:K-defin}, is an invariant and does not depend on the choice of $K$ (see Proposition 3 in \cite{St-23}). Another main result of Steinerberger (see Theorem 1 in \cite{St-23}) is a Bonnet-Myers Theorem, stating for non-negatively curved finite graphs $G$ admitting solutions for \eqref{eq:K-defin} that
$$ \rm{diam}(G) \le \frac{2n}{\Vert K \Vert_{\ell_1}} \le \frac{2}{K_0}, $$
where $K_0 = \min_i K_i \ge 0$. Moreover, Bonnet-Myers sharp graphs, that is, graphs admitting solutions for \eqref{eq:K-defin} and satisfying
$$ K_0 = \min_i K_i = \frac{2}{\rm{diam}(G)} $$
must have constant Steinerberger curvature $K_i = K_0$ for all $i$. (Note that $G$ has constant curvature if $D \mathbf{1}_n$ is a constant vector, and that all vertex transitive graphs $G$ have this property.) These and other results are derived in \cite{St-23} from a particular distance balancing property formulated as Theorem 4 (Minimax theorem). An alternative proof of the invariance of the total curvature of graphs with non-negative Steinerberger curvature without the use of the Minimax theorem was given in \cite[Theorem 3]{ChT-23}. For the reader's convenience, we provide at the beginning of Section \ref{sec:bonnetmyers} a direct proof of the Bonnet-Myers inequality and its implication of constant curvature without using the Minimax theorem. Regarding the Steinerberger curvature of the vertices of a graph $G=(V,E)$, we will use interchangeably the entries $K_i$ of the curvature vector $K$ if the vertices $V$ are enumerated, or the notation $K(x)$ for the Steinerberger curvature of an arbitrary vertex $x \in V$.

One of the appeals of Steinerberger curvature is that it has a very simple definition, is easy to compute, and is similar to Ollivier Ricci curvature or Lin-Lu-Yau curvature in various examples. See \cite{Ol-09} and \cite{LLY-11} for their definitions. In contrast to the latter two curvatures, it is a non-local curvature notion. Other curvatures related to Steinerberger's curvature were henceforth developed and investigated in \cite{DL-22} and in \cite{DOS-24}.   

Note that there is a freely accessible interactive app \cite{CKLLS-22} for the computation of Steinerberger curvature (and various other graph curvatures) at

\medskip

\begin{center}
\url{https://www.mas.ncl.ac.uk/graph-curvature/}
\end{center}

\medskip

\noindent
We encourage readers to use this tool to discover various other interesting properties of graph curvatures. 

Let us now briefly discuss the results about Steinerberger curvature presented in this note. Explicit formulas for Steinerberger curvatures of block graphs are given in Theorem \ref{thm:curvblockgraphs} in Section \ref{sec:treesandblockgraphs}. In Section \ref{sec:bridges}, we present some interesting Steinerberger curvature properties of two graphs connected via a single edge (bridge). The main results there are Theorems \ref{thm:bertie-blossom}, \ref{thm:BB_ext2} and \ref{thm:BB_ext}. 
We like to emphasize that curvature properties of graphs connected by a bridge were already studied earlier in \cite{ChT-23} (bridging and cutting graphs). In our formulation, we focus on the explicit curvature formulas as opposed to the preservance of curvature non-negativity, which  
was the main focus in \cite[Theorems 2 and 4]{ChT-23}. Section \ref{sec:problems} lists some stimulating problems related to non-positive Steinerberger curvature. The main result in the final Section \ref{sec:bonnetmyers} 
is concerned with Bonnet-Myers sharpness. It was already mentioned in \cite{St-23} that examples of Bonnet-Myers sharp graphs are even cycles, hypercubes and the Johnson graphs $J_{2n,n}$. A proof that Cartesian products of Bonnet-Myers sharp graphs are again Bonnet-Myers sharp was given in \cite[Proposition in Section 2.1]{ChT-23}. Our Theorem \ref{thm:BM-equiv} states that self-centered Bonnet-Myers sharp graphs are precisely the antipodal graphs.  
We also an inequality between the average distance of a non-negatively curved graph and its curvature (Proposition \ref{prop:avdistcurv}). 
This last section provides also some context for these result by mentioning related results based on Ollivier Ricci curvature instead of Steinerberger curvature.

\section{Block graphs}
\label{sec:treesandblockgraphs}

Algebraic investigations of distance matrices of particular graphs have been carried out over time by various authors. One of the first papers in this direction was Graham/Lov\'asz \cite{GL-78}, which contains a formula for the inverse of the distance matrix of a tree. Of particular relevance in this section 
is the paper by Bapat/Sivasubramanian \cite{BS-11}, which is concerned with 
the inverses of the distance matrices of block graphs. Note that trees are special cases of block graphs.

Let us start by giving the definition of block graphs. A maximal $2$-connected subgraph of a graph is called a \emph{block} of the graph. A \emph{block graph} is a connected graph all of whose blocks are cliques. See \cite[Figure 1]{BS-11} for an example, whose Steinerberger curvature is discussed in Example \ref{example:blockgraph} below. Trees are special cases of block graphs whose blocks are precisely the edges. Block graphs have the following Steinerberger curvatures, which is a direct consequence of Lemma 2 in \cite{BS-11}.

\begin{thm} \label{thm:curvblockgraphs}
Let $G=(V,E)$ be a block graph consisting of $r$ blocks, each of size $(P_i)_{i=1}^r$ and
$$ \lambda_G = \sum_{i=1}^r \frac{P_i-1}{P_i}. $$
Let $x \in V$ be contained in $s$ blocks labeled by $i_1,\dots,i_s$. Let
$$ \beta_x = \left( \sum_{j=1}^s \frac{1}{P_{i_j}}\right) -(s-1). $$
Then we have
$$ K(x) = \frac{|V| \beta_x}{\lambda_G}. $$
\end{thm}

\begin{proof} 
    Let $n= |V|$ and $D$ be the distance matrix of $G$. We use the same notation as in \cite{BS-11}. Theorem 3 in \cite{BS-11} tells us that $D$ is invertible. (In fact, it also provides an explicit formula for $D^{-1}$.)  Lemma 2 in \cite{BS-11} provides us with the identity
    $$ D \beta = \lambda_G {\bf{1}}_n. $$
    Therefore, the unique Steinerberger curvature vector of $G$ is given by
    $$ K = \frac{n}{\lambda_G} \beta. $$
\end{proof}

\begin{ex}\label{example:blockgraph}
Let $G=(V,E)$ be the graph with $11=|V|$ vertices, as illustrated in Figure \ref{fig:blockgraph}. This example was presented and analysed in \cite[p. 1394]{BS-11}. We have $4$ blocks and 
$$ \lambda_G = \frac{3}{4} + \frac{1}{2} + \frac{4}{5} + \frac{2}{3} = \frac{163}{60}. $$
\begin{figure}[h!]
\begin{center}
\includegraphics[width=0.8\textwidth]{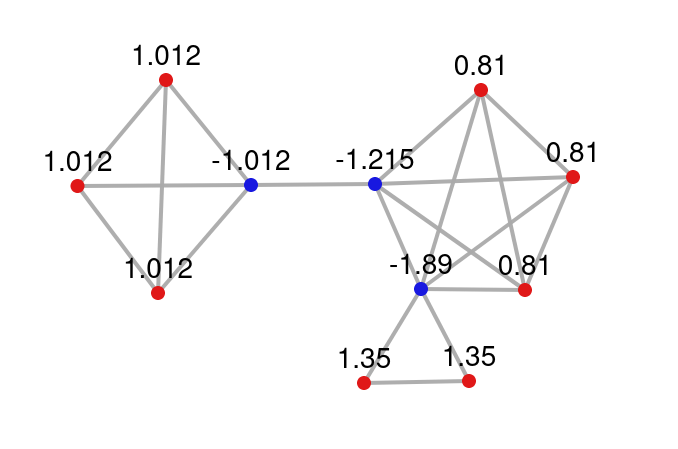}
\end{center}
\caption{Steinerberger curvature of the block graph illustrated in Figure 1 of \cite{BS-11}}
\label{fig:blockgraph}
\end{figure}
Denoting the joint vertex of the $K_5$ block and the $K_3$ block by $x$, we obtain
$$ \beta_x = \left( \frac{1}{5} + \frac{1}{3} \right) - 1 = - \frac{7}{15}. $$
This implies that $x$ has the following Steinerberger curvature:
$$ K(x) = \frac{|V|\beta_x}{\lambda_G} = -11 \cdot \frac{7}{15} \cdot \frac{60}{163} = -1.8896\dots $$
in accordance with the value given in Figure \ref{fig:blockgraph}.
\end{ex}

\begin{rmk}
The formula for Steinerberger curvature of vertices in a tree is already implicitely stated in \cite[Section 2.2]{ChT-23}, when they refer to Bapat's work \cite{Ba-10}. Steinerberger curvature of tree vertices can be also easily derived from the above result about block graphs:

For a tree $G=(V,E)$, we have $|V|-1$ blocks of size $P_i=2$ (edges). Therefore, 
$$ \lambda_G = \frac{1}{2} (|V|-1). $$
Moreover, for $x \in V$, $s$ is equal to the degree $d_x$ of the vertex $x$, and
$$ \beta_x= \frac{d_x}{2} - (d_x-1) =1 - \frac{d_x}{2}. $$
This implies that
$$ K(x) = \frac{|V| (1-d_x/2)}{(|V|-1)/2} = \frac{|V|}{|V|-1} (2-d_x). $$
\end{rmk}

\section{Bridges}
\label{sec:bridges}

In this section we discuss some amusing curvature relations between two graphs and the graph obtained by connecting them via a bridge. There is overlap with the earlier paper \cite{ChT-23} by Chen and Tsui. While Theorems 2 and 4 in \cite{ChT-23} focus on non-negativity properties of Steinerberger curvature, we consider explicit curvature relations, many of which can be extracted from the proofs given in \cite{ChT-23}. Hence we use their notation and proof techniques in explaining our deductions.

Before stating our first result, we introduce the total curvature of a subset of the vertices.
For a given finite combinatorial graph $G=(V,E)$ with Steinerberger curvature $K: V \to \mathbb{R}$, the \emph{total curvature} of a subset $W \subset V$ is defined as $K(W) := \sum_{w \in W} K(w)$. 

\begin{thm} \label{thm:bertie-blossom}
  Let $\{u,v\} \in E$ be a bridge of a finite connected graph $G=(V,E)$ separating $G$ into two components $G_1$ and $G_2$ with vertex sets $V_1$ 
  and $V_2$ 
  and $u \in V_1$ and $v \in V_2$. 
  Assume that $G$ admits a unique solution for \eqref{eq:K-defin} and let $K_G: V \to \R$ denote its Steinerberger curvature. Then the total curvature $K_G(V_1) = \sum_{x \in V_1} K_G(x)$ of the component $G_1$ agrees with the total curvature $K_G(V_2) = \sum_{y \in V_2} K_G(y)$ of the component $G_2$.
\end{thm}

\begin{proof}
    Let $n = |V|$. Since
    $$ D K_G = n \cdot {\bf{1}}_n, $$
    we have for any $z \in V$
    $$ n = \left( \sum_{x \in V_1} d(x,z) K_G(x) \right) + \left( \sum_{y \in V_2} d(y,z) K_G(y) \right). $$
    Applying this identity twice with $z=u$ and $z=v$ as well as the fact that the distance from $u$ to any vertex in $V_i$ differs by $(-1)^i$ to the distance from $v$ to that vertex, we obtain
    \begin{multline*}
        n = \left( \sum_{x \in V_1} d(x,u) K_G(x) \right) + \left( \sum_{y \in V_2} d(y,u) K_G(y) \right) \\= \left( \sum_{x \in V_1} (d(x,v)-1) K_G(x) \right) + \left( \sum_{y \in V_2} (d(y,v)+1) K_G(y) \right) \\ = n - \left( \sum_{x \in V_1} K_G(x) \right) + \left( \sum_{y \in V_2} K_G(y) \right). 
    \end{multline*}
    This implies the following agreement of total curvatures
    $$ \sum_{x \in V_1} K_G(x) = \sum_{y \in V_2} K_G(y). $$
\end{proof}

\begin{rmk}
    In fact, this result can also be found implicitely in \cite[proof of Theorem 4]{ChT-23} as the derived identity $\i^\top \w_1 = \i^\top \w_2$, where $\w_1$ and $\w_2$ are the Steinerberger curvature vectors of the vertices $V_1$ and $V_2$, but the proof given here focusses directly on this result and is much easier. 
\end{rmk}

It follows from Theorem \ref{thm:bertie-blossom}, in particular, that the curvature of any added leaf agrees with the total curvature of the rest of the new graph. But this does not imply that the curvature of any added leaf is independent of the choice of vertex where this leaf is added. However, this independence holds true under some specific conditions, which is a direct consequence of formula \eqref{eq:KGq-2} in the following Theorem. 

\begin{thm} \label{thm:BB_ext2}
    Let $G_1 = (V_1,E_1)$ with $n = |V_1| \ge 2$ and $u \in V_1$. Let $v$ be a single isolated vertex. Construct $G$ from $G_1$ by connecting $u$ and $v$ by an edge. Let $D_G, D_1$ and $K_G, K_{G_1}$ be the distance matrices and Steinerberger curvatures of the graphs $G, G_1$, respectively. Let $k_1 = K_{G_1}(V_1)$ be the total curvature of $G_1$.  Assume that the following three conditions are satisfied:
    \begin{itemize}
    \item[(C1)] We have 
    $$
    2 + \frac{k_1}{n} \neq 0; $$
    \item[(C2)] $D_1$ is invertible;
    \item[(C3)] $k_u = K_{G_1}(u) \neq 0$.
    \end{itemize}
    Then there exists a constant $\alpha \in \mathbb{R}$ such that, for all vertices $x \in V_1 \backslash \{ u \}$,
    $$
    K_G(x) = \alpha K_{G_1}(x).
    $$
The factor $\alpha$ is given by
\begin{equation}
\label{eq:alpha-2}
\alpha = \frac{2(n+1)}{2n+k_1}.
\end{equation}
The curvatures of the bridge vertices $u,v$ are given by
\begin{align} 
K_G(u) &= \gamma k_u, \nonumber\\
K_G(v) &= \frac{(n+1)k_1}{2n+k_1}, \label{eq:KGq-2}
\end{align}
where
$$
\gamma = \left( 1 - \frac{k_1}{2k_u} \right) \alpha. 
$$
\end{thm}

\begin{rmk}
Theorem \ref{thm:BB_ext2} can be viewed as a supplementary result to the later Theorem \ref{thm:BB_ext}. Theorem \ref{thm:BB_ext} below is related to Theorem 2 in the paper \cite{ChT-23} (which focusses on some different curvature aspects). Theorem \ref{thm:BB_ext2} is not covered by Theorem \ref{thm:BB_ext} below, since a single isolated vertex does not have finite Steinerberger curvature. In this case, the distance matrix is simply $D = (0)$ and the equation $DK= \i_1 =(1)$ would lead to infinite Steinerberger curvature for this isolated vertex. For that reason, a separate proof is required for this case.
\end{rmk}

\begin{proof}[Proof of Theorem \ref{thm:BB_ext2}]
We label the vertices in $G_1$ by $V_1=\{u_1,\dots,u_n\}$ and assume that $u=u_n$. The distance matrices of $G_1$ and $G$ are given by
$$ D_1 = \left(\begin{array}{cc}
		\hat{D_1} & \uu \\
		\uu^\top & 0
	\end{array}\right) \quad \text{and} \quad  D_G = \left[\begin{array}{ccc|c|c}
		& & & & \uu\\
		& \hat{D_1} & & \uu & + \\
		& & & & \i_{n-1} \\ \hline
		& \uu^\top & & 0 & 1 \\ \hline
		& \uu^\top+\i_{n-1}^\top & & 1 & 0 
	\end{array}\right],
$$
where $\uu:=d(u,\cdot) \in \R^{n-1}$  is the vector representing the distance from $u$ to the other vertices in $G_1$. Note that all entries of $\uu$ are non-zero. The entries of $D_G$ are derived by the fact that a shortest path from a vertex $x$ of $G_1$ to the vertex $v$ is a concatenation of a shortest path from $x$ to $u$ and the edge $u \sim v$. Let $k_u=K_{G_1}(u)$ be the curvature of the vertex $u$ in the original graph $G_1$. Let $\w_1 \in \R^n$ and $\hat \w_1 \in \R^{n-1}$ denote the curvatures of all vertices of $G_1$ and the one excluding $u$, respectively (that is, $\w_1=(\hat \w_1^\top,k_u)^\top \in \R^n$). It follows from (C2) that
$$ n \i_n = D_1 \cdot \w_1 = \left(\begin{array}{cc}
		\hat{D_1} & \uu \\
		\uu^\top & 0
	\end{array}\right) \left(\begin{array}{c}
	\hat{\w}_1 \\ k_u
	\end{array}\right). $$
This equation yields
\begin{align}
	\hat{D_1}\hat{\w}_1+ k_u\uu &= n\i_{n-1}, \label{eq:help1}\\
  \uu^\top \hat{\w}_1 &= n. \label{eq:help2}
\end{align}
We need to find uniquely determined constants $\alpha,\gamma,K_G(v)$ such that
$$ D_G \cdot \begin{pmatrix} \alpha \hat \w_1 \\ \gamma k_u \\ K_G(v) \end{pmatrix} = (n+1)\i_{n+1}. $$
The invertibility of $D_1$, given by condition (C2), implies the invertibility of $D_G$, by Theorem 5 in \cite{ChT-23}. Therefore, this equation nails down the Steinerberger curvatures of the vertices in $G$. Moreover, this equation translates into:
\begin{align}
	\alpha \hat{D_1} \hat{\w}_1 +\gamma k_u\uu  + K_G(v)(\uu+\i_{n-1}) &= (n+1)\i_{n-1} \label{eq:bigD-1-2}\\
	\alpha \uu^\top \hat{\w}_1 + K_G(v) &=n+1 \label{eq:bigD-2-2}\\
	\alpha (\uu^\top+\i_{n-1}^\top) \hat{\w}_1 + \gamma k_u &=n+1 \label{eq:bigD-3-2}.  
\end{align}
Computing $\eqref{eq:bigD-1-2}-\eqref{eq:bigD-2-2}\i_{n-1}$ by using \eqref{eq:help1} and \eqref{eq:help2} leads to
\begin{equation}
	((\gamma-\alpha)k_u+ K_G(v))\uu = \mathbf{0}_{n-1}. \label{eq:bigD-5-2}
\end{equation}
Using the facts that the total curvature of $V_1$ is given by $\i_{n-1}^\top \hat{w}_1 + k_u$ and that $\uu$ is a non-zero vector, we can rewrite the equations \eqref{eq:bigD-2-2},\eqref{eq:bigD-3-2},\eqref{eq:bigD-5-2} as follows:
$$ \left(\begin{array}{ccc}
		n & 0 & 1    \\
		n+k_1-k_u & k_u & 0    \\
		-k_u & k_u & 1
	\end{array}\right)
\left(\begin{array}{c}
	\alpha \\ \gamma \\ K_G(v)
\end{array}\right) = 
\left(\begin{array}{c}
	n+1 \\ n+1 \\ 0
\end{array}\right). $$
The determinant of the matrix on the left hand side is $k_u(k_1+2n) \neq 0$ by (C1) and (C3). Therefore, the matrix on the left hand side can be inverted, leading to a unique solution, given by
\begin{align*}
	\left(\begin{array}{c}
		\alpha \\ \gamma \\ K_G(v)
	\end{array}\right) 
	= \frac{n+1}{k_1+2n}
	\left(\begin{array}{c}
		2 \\ 2(1-\frac{k_1}{2k_u}) \\ 
		k_1
	\end{array}\right).
\end{align*}
This finishes the proof of the theorem.
\end{proof}

The following theorem is closely related to a result and proof in \cite{ChT-23}. While Theorem 2 in \cite{ChT-23} is concerned with non-negativity of Steinerberger curvature of two non-negatively curved graphs connected via a bridge, a slight modification of its proof leads to the theorem presented here. 
 
\begin{thm}[cf. {\cite[Theorem 2]{ChT-23}} and proof therein] \label{thm:BB_ext}
    For $i \in \{1,2\}$, let $G_i = (V_i,E_i)$ and $u \in V_1$ and $v \in V_2$. Let $n_1 = |V_1| \ge 2$ and $n_2=|V_2| \ge 2$. Construct $G$ by connecting $G_1$ and $G_2$ by an edge $u \sim v$. Let $D_G, D_1, D_2$ and $K_G, K_{G_1},K_{G_2}$ be the distance matrices and Steinerberger curvatures of the graphs $G, G_1, G_2$, respectively.  Let $k_1 = K_{G_1}(V_1)$ be the total curvature $G_1$ and $k_2=K_{G_2}(V_2)$ be the total curvature of $G_2$.  Assume that the following three conditions are satisfied:
    \begin{itemize}
    \item[(C1)] We have 
    $$
    Z = \left( 2 + \frac{k_1}{n_1} \right)\left( 2 + \frac{k_2}{n_2} \right) \neq 4; $$
    \item[(C2)] $D_1$ and $D_2$ are invertible;
    \item[(C3)] $k_u = K_{G_1}(u) \neq 0$ and $k_v = K_{G_2}(v) \neq 0$.
    \end{itemize}
    Then there exists constants $\alpha,\beta \in \mathbb{R}$ such that, for all vertices $x \in V_1 \backslash \{ u \}$ and $y \in V_2 \backslash \{v\}$,
    $$
    K_G(x) = \alpha K_{G_1}(x), \qquad
    K_G(y) = \beta K_{G_2}(y).
    $$
The factors $\alpha, \beta$ are given by
\begin{equation} \label{eq:alphabeta}
\alpha = \frac{2(n_1+n_2)k_2}{n_1n_2(Z-4)}
\qquad
\beta = \frac{2(n_1+n_2)k_1}{n_1n_2(Z-4)}. 
\end{equation}
The curvatures of the bridge vertices $u,v$ are given by
$$ K_G(u) = \gamma K_{G_1}(u), \quad K_G(v) = \delta K_{G_2}(v) $$ 
with
$$
\gamma = \left( 1 - \frac{k_1}{2k_u} \right) \alpha, 
\qquad
\delta =  \left( 1 - \frac{k_2}{2k_v} \right) \beta.
$$
\end{thm}

\begin{rmk} Note that Steinerberger curvatures of the bridge vertices $u$ and $v$ in $G$ can also be found in \cite[(2.1) and (2.2)]{ChT-23}.
\end{rmk}

\begin{proof}
    We follow closely the notation and arguments in the proof of \cite[Theorem 2]{ChT-23} with the necessary slight modifications. Let
    $$ V_1 = \{ u_1,\dots,u_{n_1}=u \}, $$ 
    $$ V_2 = \{ v=v_1,\dots,v_{n_2} \}, $$
    $\y$ be the last column of $D_1$ and $\z$ be the first column of $D_2$. Moreover, let $\w_i \in \R^{n_i}$ be the solution of 
    $$ D_i \w_i = n_i \i_{n_i} $$ 
    for $i=1,2$. Uniqueness of this solution follows from condition (C2). We also know from Theorem 5 in \cite{ChT-23} that the distance matrix $D_G$ of the graph $G$ is invertible, that is, there exists a unique solution for $D_G \w = (n_1+n_2) \i_{n_1+n_2}$, which we want to find. We set
    $$ \w = \left(\begin{array}{c} \alpha \w_1 \\ \beta \w_2 \end{array}\right) +(\gamma-\alpha)k_u e_{n_1} + (\delta-\beta)k_v e_{n_1+1}. $$
    Then an analogous computation to the one in \cite{ChT-23} yields
    \begin{align*} 
    D_G \w &= \left(\begin{array}{c} \alpha n_1 \i + \beta \left( k_2 \y + k_2 \i + n_2\i \right) \\ \alpha(k_1 \z + k_1 \i + n_1 \i) + \beta n_2 \i \end{array} \right)+ \left(\begin{array}{c} (\gamma-\alpha)k_u \y + (\delta-\beta)k_v (\y + \i)  \\ (\gamma-\alpha)k_u (\z+\i) + (\delta-\beta)k_v \z \end{array}\right) \\
    &= \left(\begin{array}{c} (\alpha n_1 + \beta n_2 + \beta k_2 + (\delta-\beta)k_v) \i + ((\gamma-\alpha)k_u+(\delta-\beta)k_v+\beta k_2) \y \\ (\alpha n_1 + \beta n_2 + \alpha k_1 + (\gamma-\alpha)k_u) \i + ((\gamma-\alpha)k_u+(\delta-\beta)k_v+\alpha k_1) \z \end{array}\right),
    \end{align*}
    where we used $\i^\top \w_i = k_i$ for $i=1,2$ and $y^\top \w_1 = n_1$ and $z^\top \w_2 = n_2$. In order to match the right hand side with $(n_1+n_2)\i_{n_1+n_2}$, the coefficients on the right hand in front of $\i$ must coincide with $n_1+n_2$ and the coefficients in front of $\y$ and $\z$ must vanish:
    \begin{align*}
    n_1 \alpha + (n_2+k_2-k_v)\beta +k_v \delta &= n_1+n_2, \\
    (n_1+k_1-k_u) \alpha + n_2 \beta + k_u \gamma &= n_1+n_2, \\
    -k_u \alpha +(k_2-k_v)\beta + k_u \gamma + k_v \delta &= 0, \\
    (k_1-k_u)\alpha-k_v\beta + k_u \gamma + k_v \delta &= 0. 
    \end{align*}
    This can be rewritten as 
    \begin{align*}
	\left(\begin{array}{cccc}
		n_1 & n_2+k_2-k_v & 0 & k_v    \\
		n_1+k_1-k_u & n_2 & k_u & 0    \\
		-k_u & k_2-k_v & k_u & k_v \\
		k_1-k_u & -k_v & k_u & k_v
	\end{array}\right)
\left(\begin{array}{c}
	\alpha \\ \beta \\ \gamma \\ \delta
\end{array}\right) = 
\left(\begin{array}{c}
	n_1+n_2 \\ n_1+n_2 \\ 0 \\ 0
\end{array}\right).
\end{align*}
The determinant of the matrix on the left hand side is $-n_1 n_2 k_u k_v (Z-4) \neq 0$ by (C1) and (C3). Therefore, this equation has a unique solution given  by
\begin{align*}
	\left(\begin{array}{c}
		\alpha \\ \beta \\ \gamma \\ \delta
	\end{array}\right) 
	= \frac{n_1+n_2}{n_1n_2(Z-4)}
	\left(\begin{array}{c}
		2k_2 \\ 2k_1 \\ 
		2k_2(1-\frac{k_1}{2k_u}) \\ 2k_1(1-\frac{k_2}{2k_v})
	\end{array}\right).
\end{align*}
This finishes the proof of the theorem.
\end{proof}

Let us end this section with the following corollary.

\begin{cor}
    With the assumptions and notations of Theorem \ref{thm:BB_ext}, we have
    \begin{equation} \label{eq:KGVrel} 
    K_G(V) = \frac{\alpha}{2} K_{G_1}(V_1) + \frac{\beta}{2} K_{G_2}(V_2) = \alpha K_{G_1}(V_1) = \beta K_{G_2}(V_2). 
    \end{equation}
    In particular, the total Steinerberger curvature of the graph $G$ is determined by the sizes and total
    Steinerberger curvatures of $G_1$ and $G_2$ and does not depend on which vertices of $G_1$ and $G_2$ are connected by a bridge to construct $G$.
\end{cor}

\begin{proof}
    Let $u$ be a vertex in $G_1=(V_1,E_1)$, $v$ be a vertex in $G_2=(V_2,E_2)$, and $G=(V=V_1\cup V_2,E)$ be constructed by connecting $G_1$ and $G_2$ by the edge $u \sim v$, that is $E = E_1 \cup E_2 \cup \{\{u,v\}\}$.
    If follows from Theorem \ref{thm:BB_ext} that
    \begin{align*}
    K_G(V_1) &= K_G(V_1\setminus \{u\}) + K_G(u) \\
    &= \alpha K_{G_1}(V_1\setminus \{u\}) + \gamma k_u \\
    &= \alpha K_{G_1}(V_1\setminus \{u\}) + \left( 1 - \frac{k_1}{2k_u} \right) \alpha k_u \\
    &= \alpha K_{G_1}(V_1\setminus \{u\}) + \alpha K_{G_1}(u) - \frac{\alpha K_{G_1}(V_1)}{2} \\
    &= \frac{\alpha}{2} K_{G_1}(V_1). 
    \end{align*}
    Combining this result with the same result for $K_G(V_2)$ confirms the first equality of \eqref{eq:KGVrel}. The last two equations follow from Theorem \ref{thm:bertie-blossom}. The final statement follows from the fact that $\alpha,\beta$ depend only on the sizes and total Steinerberger curvatures of $G_1$ and $G_2$.
\end{proof}

\section{Questions related to non-positive Steinerberger curvature}
\label{sec:problems}

The content of this section are mainly observations and questions for which we have no rigorous proofs and only experimental evidence using the freely accessible earlier mentioned graph curvature applet at

\medskip

\begin{center}
\url{https://www.mas.ncl.ac.uk/graph-curvature/}
\end{center}

\medskip

\noindent
This app allows to compute Steinerberger curvature for arbitrary combinatorial graphs.
We hope that these observation stimulate further research.  

\begin{figure}[h!]
\begin{center}
\includegraphics[width=\textwidth]{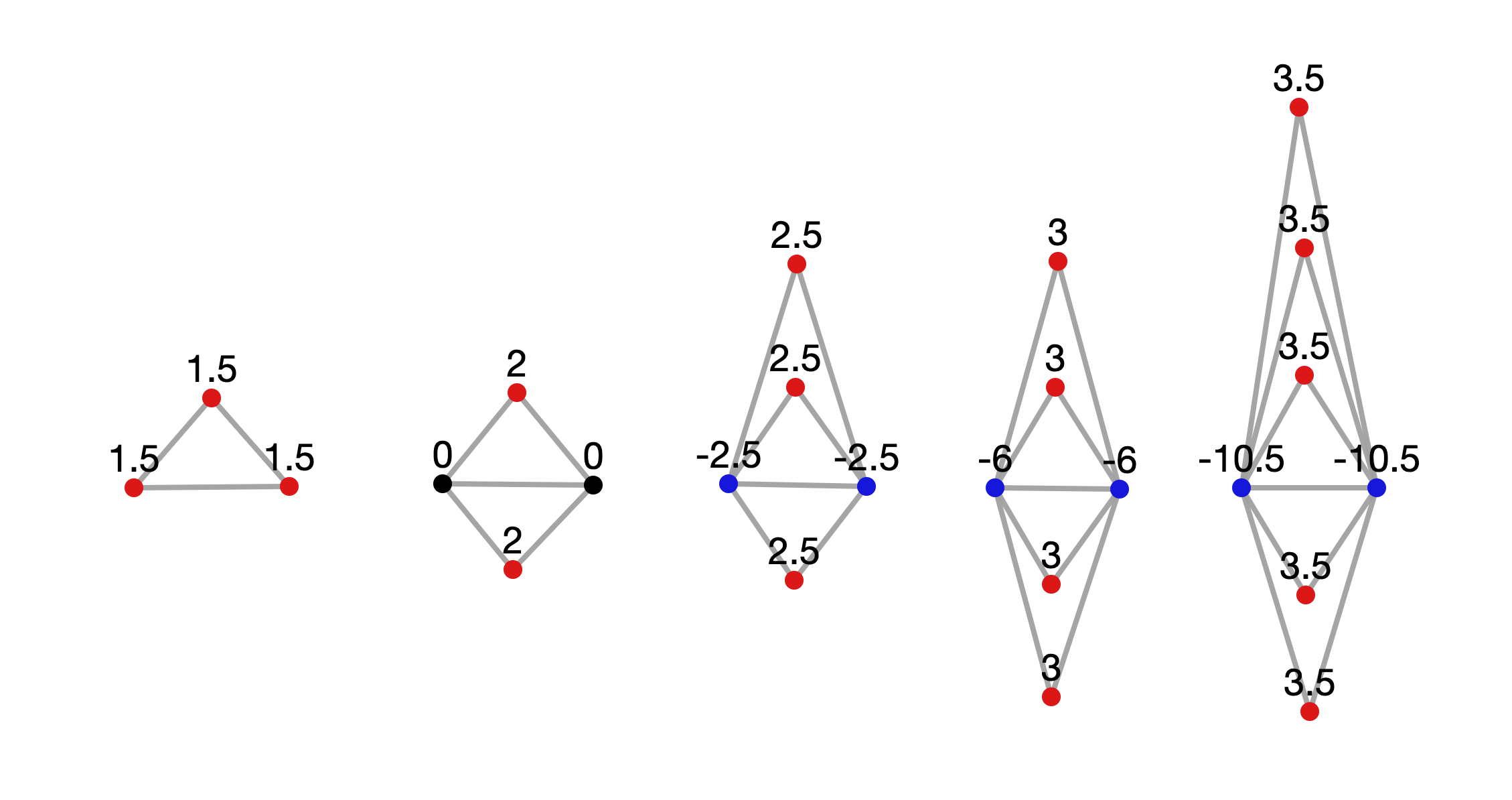}
\end{center}
\caption{The graph family $A(n)$: $A(1), A(2), A(3), A(4), A(5)$ from left to right}
\label{fig:bookgraphs}
\end{figure}

Our first investigations are concerned with total Steinerberger curvature. It is not easy to find examples of small finite graphs which have zero or negative total Steinerberger curvature. Recall that all vertex transitive graphs have automatically positive Steinerberger curvature in contrast to Ollivier Ricci curvature (see \cite{Ol-09} for the definition of the latter, which is a curvature notion defined on edges or, more generally, on pairs of vertices). While there are no expander families with non-negative Ollivier Ricci curvature with idleness $p=1/2$ (see \cite{Sa-22}), most of these families are Cayley graph families and therefore vertex transitive with positive Steinerberger curvature. Note that this fact uncovers a stark contrast between Steinerberger curvature and Ollivier Ricci curvature: almost all explicit examples of expander graph families have positive Steinerberger curvature but there must be graphs within all these families whose Ollivier Ricci curvature is negative for at least some edges.   

Figure \ref{fig:bookgraphs} presents a family of graphs $A(n)$, which are constructed by glueing $n$ triangles along a common egde. Amongst these graphs, $A(4)$ has $6$ vertices and vanishing total Steinerberger curvature and $A(5)$ has seven vertices and negative total Steinerberger curvature. The characteristic polynomial of the distance matrix of $A(n)$ is given by
$$ (x^2+(1-2n)x-2)(x+1)(x+2)^{n-1}, $$
which shows that the distance matrix of all of these graphs is non-singular. Another graph with vanishing total Steinerberger curvature and 9 vertices is presented in Figure \ref{fig:phil-example}(left). The characteristic polynomial of the distance matrix of this graph is
$$ x (x^2-14x+2) (x^2+4x+1)^2 (x+3)^2, $$
and therefore singular. These examples give rise to the following:

\medskip

{\bf Problem 1:} Find small graphs which have vanishing or negative total Steinerberger curvature. Are there any graphs with vanishing total Steinerberger curvature with less than $6$ vertices? Are there any graphs with negative total Steinerberger cuvature with less than $7$ vertices?

\medskip

\begin{figure}[h!]
\begin{center}
\includegraphics[width=0.37\textwidth]{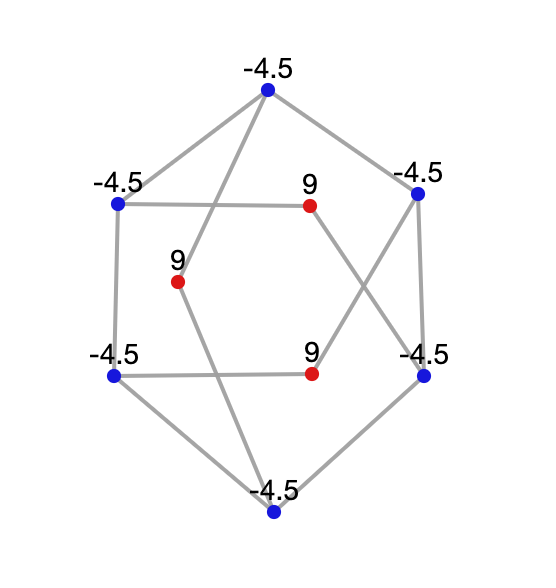}
\includegraphics[width=0.53\textwidth]{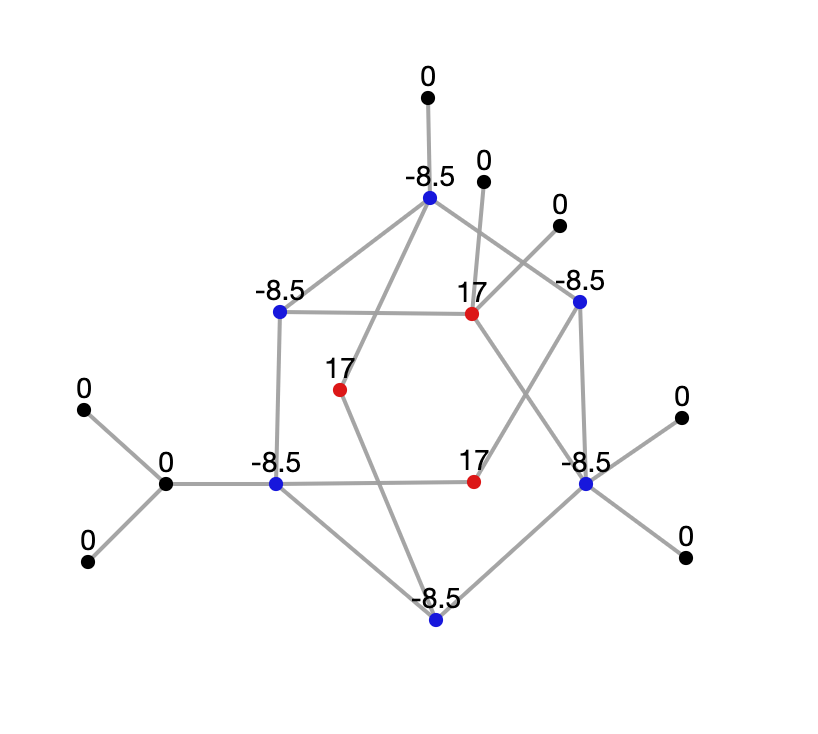}
\end{center}
\caption{A graph $G$ with 9 vertices and vanishing total Steinerberger curvature (left) and an extension of $G$ by consecutive attachment of $8$ leaves (right)}
\label{fig:phil-example}
\end{figure}

Another surprising property of the graph $G$ in Figure \ref{fig:phil-example}(left) is the following. If the graph $G$ is extended to a sequence of graphs $G_0=G, G_1, G_2, \dots, G_k$, where $G_{i+1}$ is obtained from $G_i$ by the attachment of a leaf to any of the vertices of $G_i$, then the Steinerberger curvature of the three positively curved vertices increases by $k$, that is the number of newly created vertices. For example, the graph in Figure \ref{fig:phil-example}(right) is derived from $G$ by the successive attachment of $8$ leaves and the Steinerberger curvatures of the positively curved vertices increases from $9$ to $17 = 9+8$. A similar counting property can be observed for the vertices with negative Steinerberger curvature. There the curvature drops by $1/2$ after every attachment of a new leaf. The same phenomenon can be observed when we apply this process to the graph $A(3)$ of vanishing total Steinerberger curvature. This leads to the following question.

\medskip

{\bf Problem 2:} Given a connected graph with vanishing total Steinerberger curvature. Is it true that each positively (negatively) curved vertex increases (decreases) by the same fixed amount (depending on the choice of vertex) in every step of attaching a new leaf? 

\medskip

\begin{figure}[h!]
\begin{center}
\includegraphics[height=6cm]{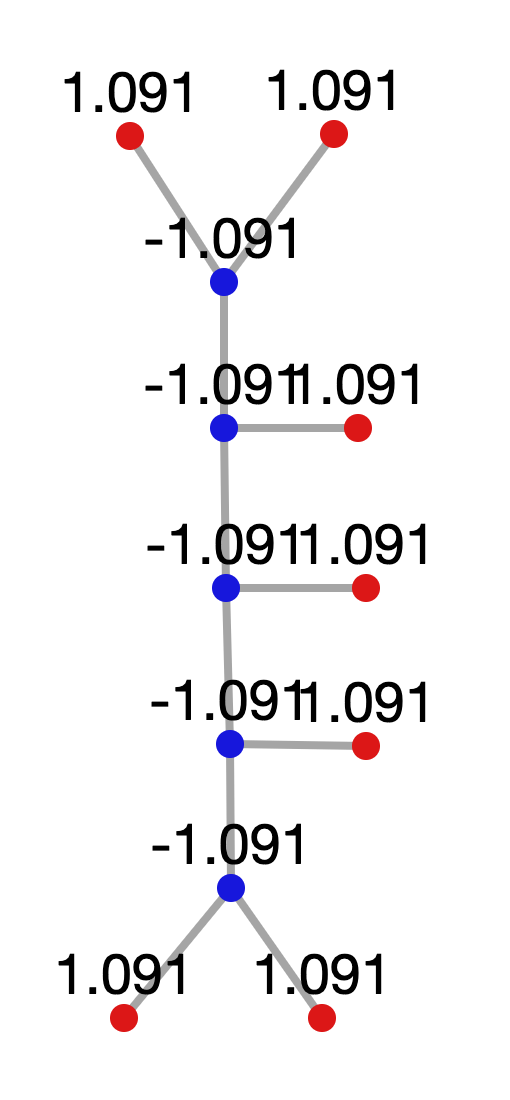}
\includegraphics[height=6cm]{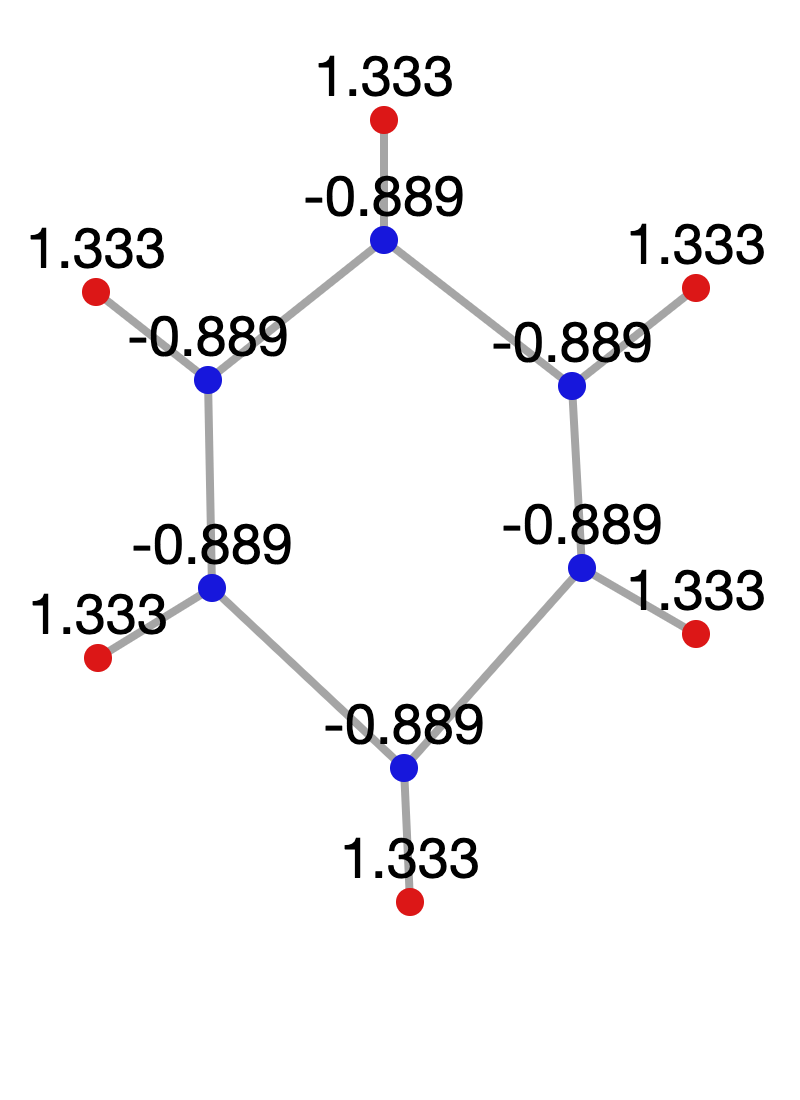}
\includegraphics[height=6cm]{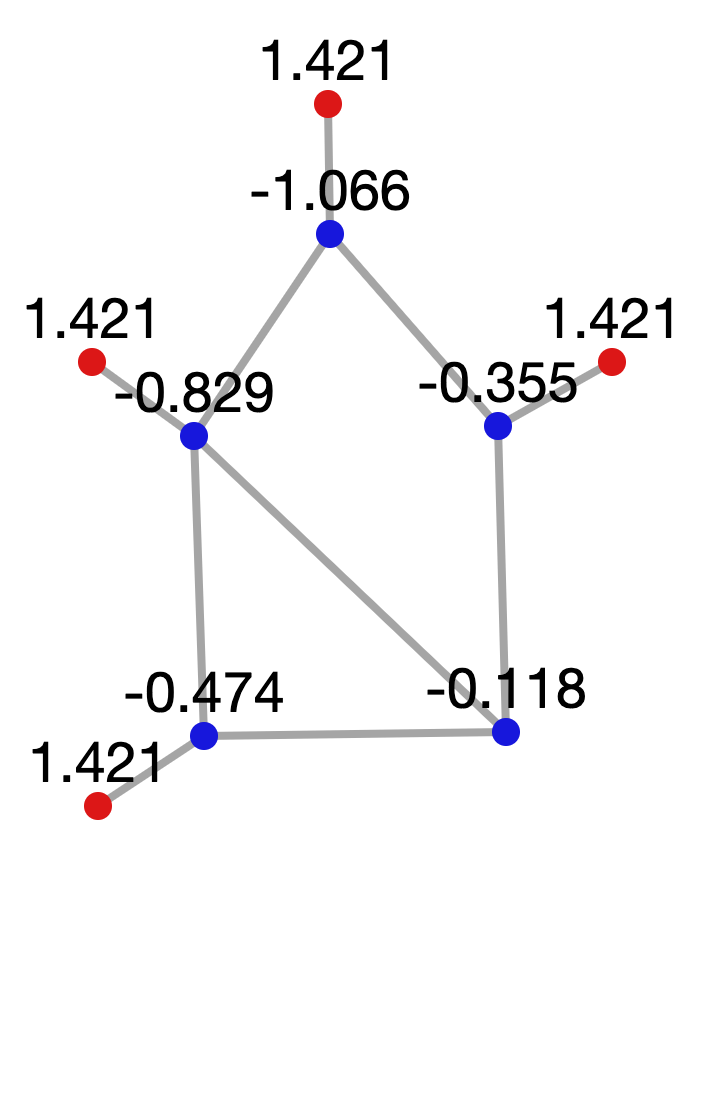}
\end{center}
\caption{A path length $5$ (left) needs $7$ leaves, a cycle of length $6$ (center) needs $6$ leaves, and the graph with $5$ vertices (right) needs $4$ leaves to make all their vertices negatively curved.}
\label{fig:make-neg}
\end{figure}

The final problem is concerned with a graph $G=(V,E)$ with strictly positive Steinerberger curvature at all vertices. Experiments show that the curvatures of all vertices $V$ of $G$ can be made negative by the attachment of a sufficient number of leaves to each of these vertices. For example, Steinerberger curvature of all vertices of a path of length $n$ can be made negative by the attachment of $n+2$ leaves (see Figure \ref{fig:make-neg}(left)), Steinerberger curvature of all vertices of a cycle of length $n$ can be made negative by the attachment of $n$ vertices (see Figure \ref{fig:make-neg}(center)), and there are also positively curved graphs with $n$ vertices whose Steinerberger curvatures can be made negative by adding less than $n$ leaves (see Figure \ref{fig:make-neg}(right)). This motivates our final problem.

\medskip

{\bf{Problem 3:}} For any positively curved graph, is there a way to determine the minimum number of leaves that need to be added to make all of its vertices negatively curved?

\medskip

\section{Bonnet-Myers sharp graphs}
\label{sec:bonnetmyers}

As mentioned in the Introduction, Steinerberger's Bonnet-Myers Theorem is a consequence of his Minimax theorem (Theorem 4 in \cite{St-23}), which is the strongest result in Steinerberger's paper. For the reader's convenience, we start with a direct proof of the inequality
\begin{equation} \label{eq:BM-simple}
{\rm{diam}}(G) \le \frac{2}{\min_i K_i} 
\end{equation}
for any solution $K \in \mathbb{R}_{\ge 0}^n$ of equation \eqref{eq:K-defin} as well as the fact that Bonnet-Myers sharp graphs have constant Steinerberger curvature. The smallest graph which has constant Steinerberger curvature $K$ and which is not Bonnet-Myers sharp is the triangle with $K=\frac{3}{2}$.
Recall that $G$ is called Bonnet-Myers sharp if equation \eqref{eq:K-defin} has a solution and if \eqref{eq:BM-simple} holds with equality. It was shown in \cite{ChT-23} that Bonnet-Myers sharpness is preserved under taking Cartesian products. 

\begin{proof}[Proof of \eqref{eq:BM-simple}]
    Let $G=(V,E)$ be a finite connected graph with
    $$ V = \{ v_1, \dots, v_n \} $$ 
    and $v_j,v_k \in V$ satisfying 
    $$ d(v_j,v_k) = L := {\rm{diam}}(G). $$
    Let $K \in \mathbb{R}^n$ be a solution of
    equation \eqref{eq:K-defin}. Then we have, by the triangle inequality and \eqref{eq:K-defin},
    \begin{equation} \label{eq:BM-proof-simple} 
    n \cdot \min_i K_i \cdot L \le \sum_i K_i (d(v_j,v_i)+d(v_i,v_k)) = n + n, 
    \end{equation}
    which implies $\min_i K_i \le 2/L$. This proves \eqref{eq:BM-simple}.

    Now we assume that
    $$ \min_i K_i = \frac{2}{L} $$
    for some solution $K$ of equation \eqref{eq:K-defin}, that is, $G$ is Bonnet-Myers sharp. (Note that, by the inequality $\min_i K_i \le 2/L$, this equation holds then for any solution $K$.) In this case, all inequalities in \eqref{eq:BM-proof-simple} are equalities and all entries $K_i$ must coincide. This shows that all Bonnet-Myers sharp graphs have constant Steinerberger curvature.
\end{proof}

For the main result in this section, we first need to introduce the notions of self-centered and antipodal graphs. These notions were also relevant in the classification result of Bonnet-Myers sharpness for Ollivier Ricci curvature (see \cite[Theorem 1.6]{CKKLMP-20}).
Note that in the definition of antipodal graphs, we follow the restrictive definition given in \cite{KMS-04}. 

\begin{defin}
  Let $G=(V,E)$ be a finite connected graph.
  
  The graph $G$ is called \emph{self-centered} if, for every vertex $x \in V$, there exists a vertex $\hat x \in V$ with $d(x,\hat x) = {\rm{diam}}(G)$. 

  The graph $G$ is called \emph{antipodal} if, for every vertex $x \in V$, there exists a vertex $\hat x \in V$ satisfying
  $$ [x,\hat x] := \{ y \in V: d(x,\hat x) = d(x,y) + d(y,\hat x) \} = V. $$
\end{defin}

Let us quickly show that any antipodal graph $G=(V,E)$ is also self-centered: Let $y,y' \in V$ be a pair of vertices satisfying $d(y,y') = L := {\rm{diam}}(G)$. If $G$ is antipodal, there exists $\hat y \in V$ such that $[y,\hat y] = V$ and it is easy to see that we must have $\hat y =y'$. Let $x \in V$ an arbitrary vertex and, since $G$ is antipodal, there exists $\hat x \in V$ satisfying:
$$ [x,\hat x] = V. $$
Then we have, by the triangle inequality
\begin{align*}
2d(x,\hat x) &= (d(x,y) + d(y,\hat x)) + (d(x,\hat y)+ d(\hat y,\hat x)) \\
&= (d(y,x) + d(x,\hat y)) + (d(y,\hat x)+d(\hat x,\hat y)) \\
& = 2d(y,\hat y) = 2L.
\end{align*}
This shows that $d(x,\hat x) = L$ and, therefore, $G$ is self-centered. \qed

\medskip

Now we are ready to prove the main result of this section, providing a large class of Bonnet-Myers sharp graphs: the antipodal graphs.

\begin{thm} \label{thm:BM-equiv}
  Let $G$ be a finite connected graph admitting solutions for \eqref{eq:K-defin}. Then the following are equivalent:
  \begin{itemize}
      \item[(a)] $G$ is antipodal.
      \item[(b)] $G$ is self-centered and Bonnet-Myers sharp.
  \end{itemize}
\end{thm}

\begin{proof}
    $(a) \Rightarrow (b)$: 
    Assume that $G$ is antipodal. Then for every $x \in V$ there exists $\hat x \in V$ such that $[x,\hat x] = V$. It is easy to see that $\hat x$ is unique and must have maximal distance to $x$. Since antipodal graphs are self-centered, we have also 
    $$ d(x,\hat x) = L := {\rm{diam}}(G). $$
    The pairings $\{ x,\hat x\}$ lead into a partition $V$ into two equal size disjoint sets $V = V_1 \cup V_2$ such that, for every $x \in V_1$, we have $\hat x \in V_2$ and vice versa. This implies for all $x \in V$,
    $$ \sum_{y \in V} d(x,y) = \sum_{y \in V_1} d(x,y) + d(x,\hat y) = \sum_{y \in V_1} L = \frac{n L}{2}. $$
    This shows that the constant vector is an eigenvector to the distance matrix. We therefore found a solution
    $$ D K = n \cdot \mathbf{1}_n $$
    with $K = \frac{2}{L} \mathbf{1}_n$. Since $\min_i \hat K_i \le 2/L$ for any solution $\hat K$ of \eqref{eq:K-defin}, $K$ is a solution whose minimal entry is equal to $2/L$ and, therefore, maximal. Consequently, all vertices of $G$ have Steinerberger curvature $2/L$ and $G$ is Bonnet-Myers sharp. 

    \medskip

    $(b) \Rightarrow (a)$: Now we assume that $G=(V,E)$ is self-centered and Bonnet-Myers sharp with $n = |V|$. Let $L := {\rm{diam}}(G)$. Then $G$ has constant Steinerberger curvature $K(z) = 2/L$ for all vertices $z \in V$. It follows from equation \eqref{eq:K-defin} that
    \begin{equation} \label{eq:BMsh-eq} 
    \frac{n L}{2} = \sum_{y \in V} d(z,y)
    \end{equation}
    for all $z \in V$.
    
    For $x \in V$, pick $\hat x$ with $d(x,\hat x) = L$. We need to show that $[x,\hat x] = V$. Applying \eqref{eq:BMsh-eq} twice and the triangle inequality, we conclude that
    $$ \frac{nL}{2} = \sum_{y \in V} d(x,y) \stackrel{(*)}{\ge} \sum_{y \in V} (L-d(\hat x,y)) = n L - \sum_{y \in V} d(\hat x,y) = nL - \frac{nL}{2} = \frac{nL}{2}. $$
    Therefore, the inequality $(*)$ is an equality and we have $d(x,y) + d(y,\hat x) = L = d(x,\hat x)$ for all $y \in V$. This finishes the proof of $[x,\hat x] = V$. 
\end{proof}

The following example shows that the condition of self-centeredness in $(b)$ of Theorem \ref{thm:BM-equiv} cannot be removed. 

\begin{ex}[Handa graph]
The following example was first considered in Handa \cite{H-99}, and it is a prominent example of a distance-balanced graph. \emph{Distance-balanced graphs} are finite connected combinatorial graphs $G=(V,E)$ with the following property for every pair $x,y \in V$ of adjacent vertices: The number of vertices in $G$ closer to $x$ agrees with the number of vertices closer to $y$ (see, e.g., \cite{IKM-10} and the references therein about this topic). 

\begin{figure}[h!]
\begin{center}
\includegraphics[width=0.54\textwidth]{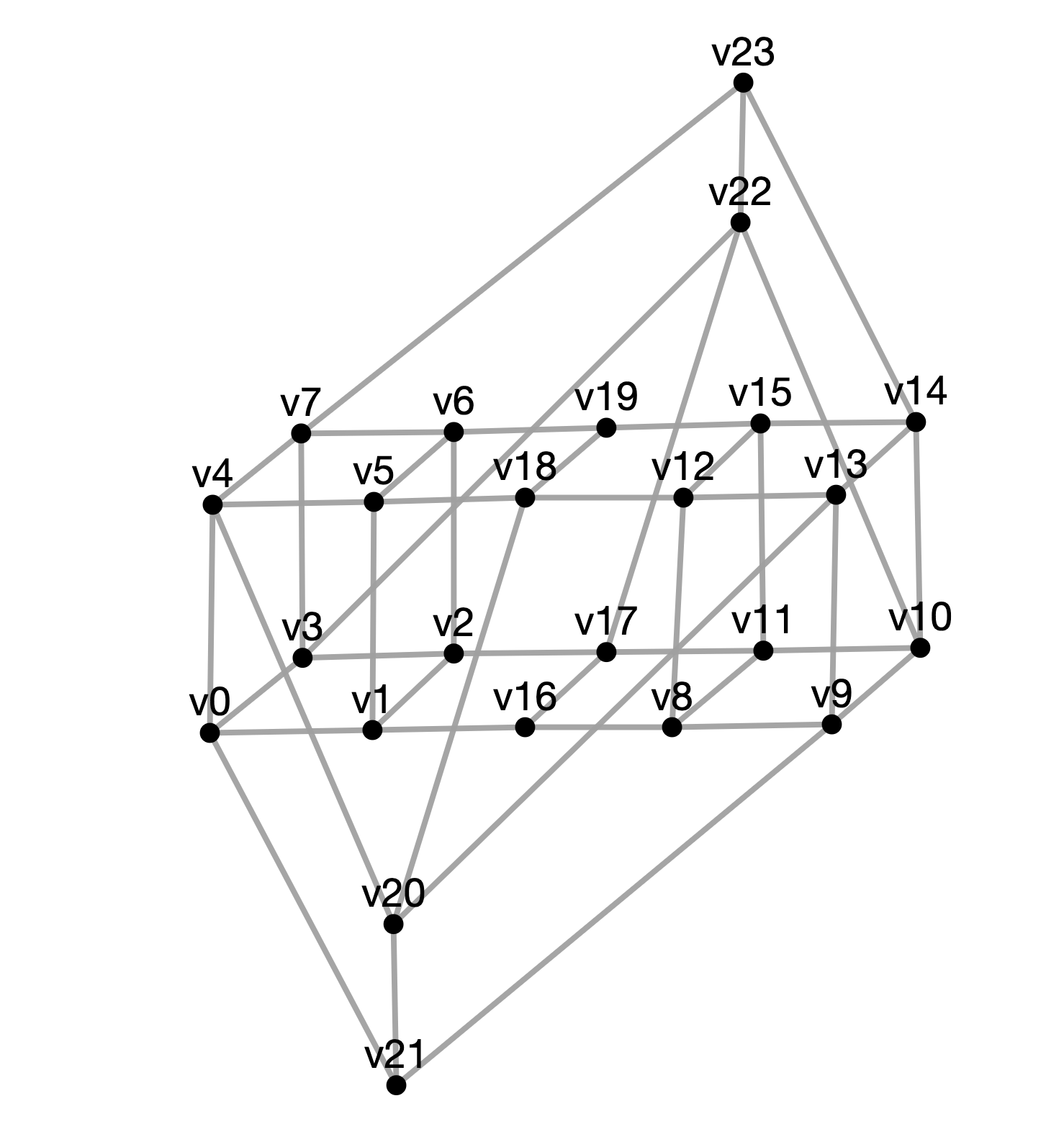}
\includegraphics[width=0.44\textwidth]{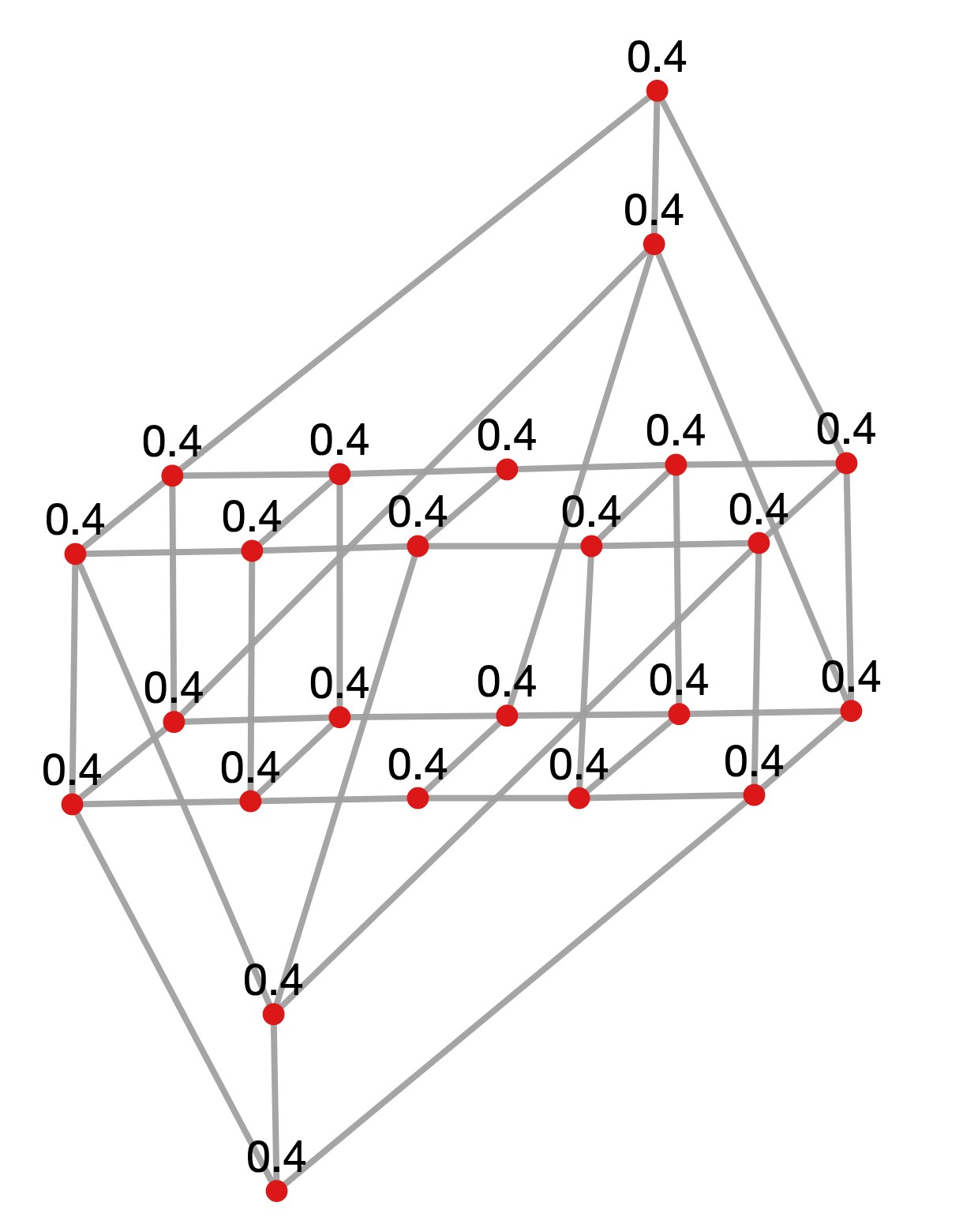}
\end{center}
\caption{The distance-balanced and Bonnet-Myers sharp Handa graph with vertex enumeration (left) and its Steinerberger curvatures (right)}
\label{fig:handa}
\end{figure}

The \emph{Handa graph} $G=(V,E)$ with
$$ V = \{v_0,v_1,\dots,v_{23}\} $$ 
is illustrated in Figure \ref{fig:handa}(left), and it is a non-regular bipartite distance-balanced graph of diameter ${\rm{diam}}(G) = 5$. The vertex pairs at distance $5$ apart are $\{v_i,v_{10-i}\}$ for $0 \le i \le 10$, $\{v_{17},v_{20}\}$ and $\{v_{18},v_{22}\}$. However, there are four vertices, $v_{16}, v_{19}, v_{21}, v_{23}$, which do not have any vertices at distance $5$. The Handa graph is Bonnet-Myers sharp with constant Steinerberger curvature $\frac{2}{5}$(see Figure \ref{fig:handa}(right), and the solution space of 
$$ D K = 24 \cdot \mathbf{1}_{24} $$
is $18$-dimensional. This is an example of a Bonnet-Myers sharp graph which is not self-centered and therefore also not antipodal. In this example, Steinerberger curvature differs from non-normalized Bakry-{\'E}mery curvature (with dimension parameter equals $\infty$) and Ollivier Ricci curvature $\kappa_p$ with idleness $p=0$, since both latter curvatures of this graph are non-positive, as illustrated in Figure \ref{fig:handa2}. For further background about Bakry-\'Emery curvature we refer readers to the original paper \cite{BE-85}, to \cite{El-91,LY-10,Schm-99} for other pioneering papers on this curvature notion, and to the more recent paper \cite{CLP-20} and its references for various properties of non-normalized Bakry-\'Emery curvature. We note however that the Lin-Lu-Yau version of Ollivier Ricci curvature, introduced in \cite{LLY-11} and related to Ollivier's original Ricci curvature $\kappa_p$ in \cite{Ol-09} via
$$ \kappa_{LLY}(x,y) = \lim_{p \to 1} \frac{\kappa_p(x,y)}{1-p} \stackrel{(*)}{=} 2 \kappa_{1/2}(x,y) $$ 
for adjacent vertices $x \sim y$ of the Handa graph has both signs. Note that the last identity $(*)$ is a consequence of \cite[Remark 5.4]{BCLMP-18}.

\begin{figure}[h!]
\begin{center}
\includegraphics[width=0.48\textwidth]{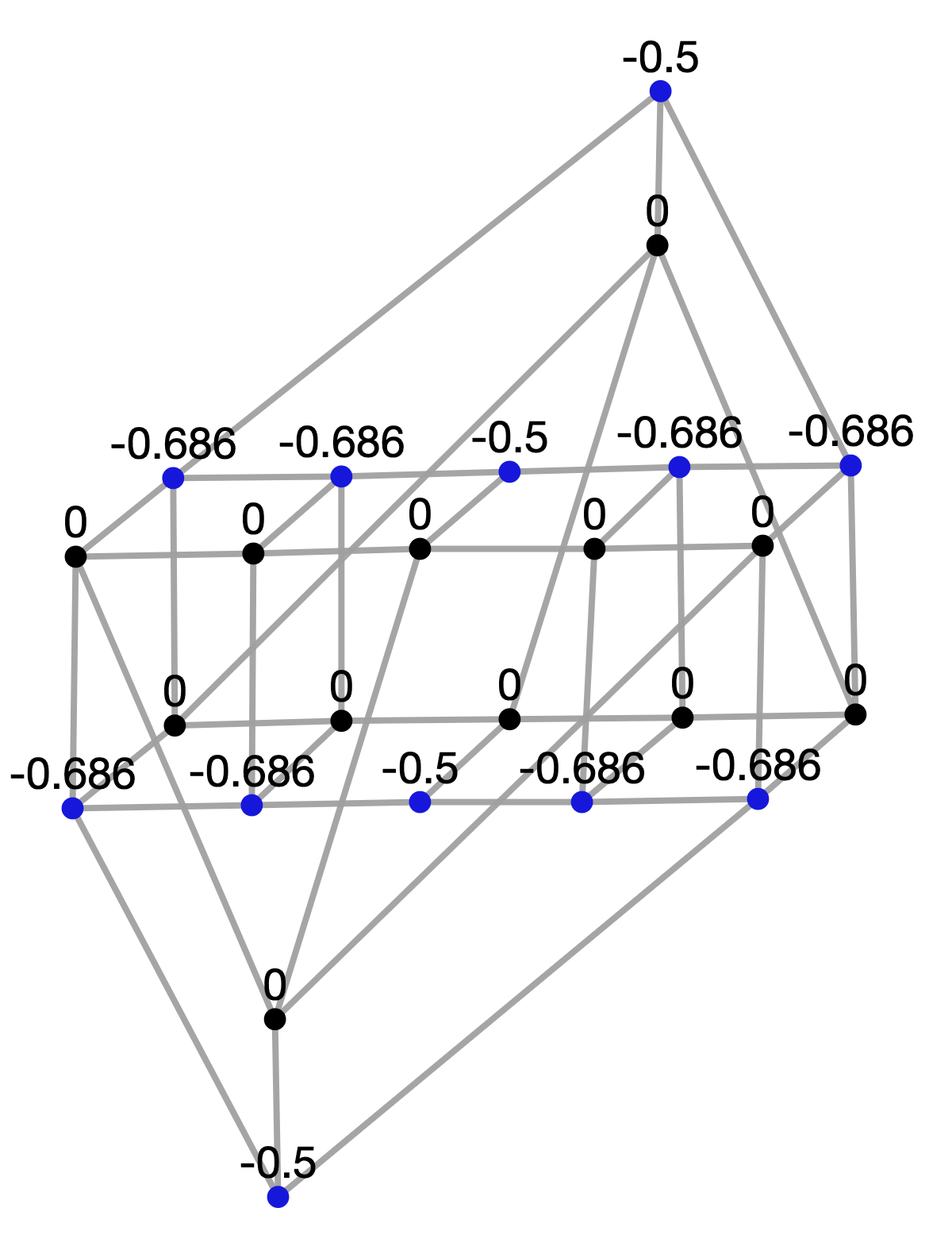}
\includegraphics[width=0.48\textwidth]{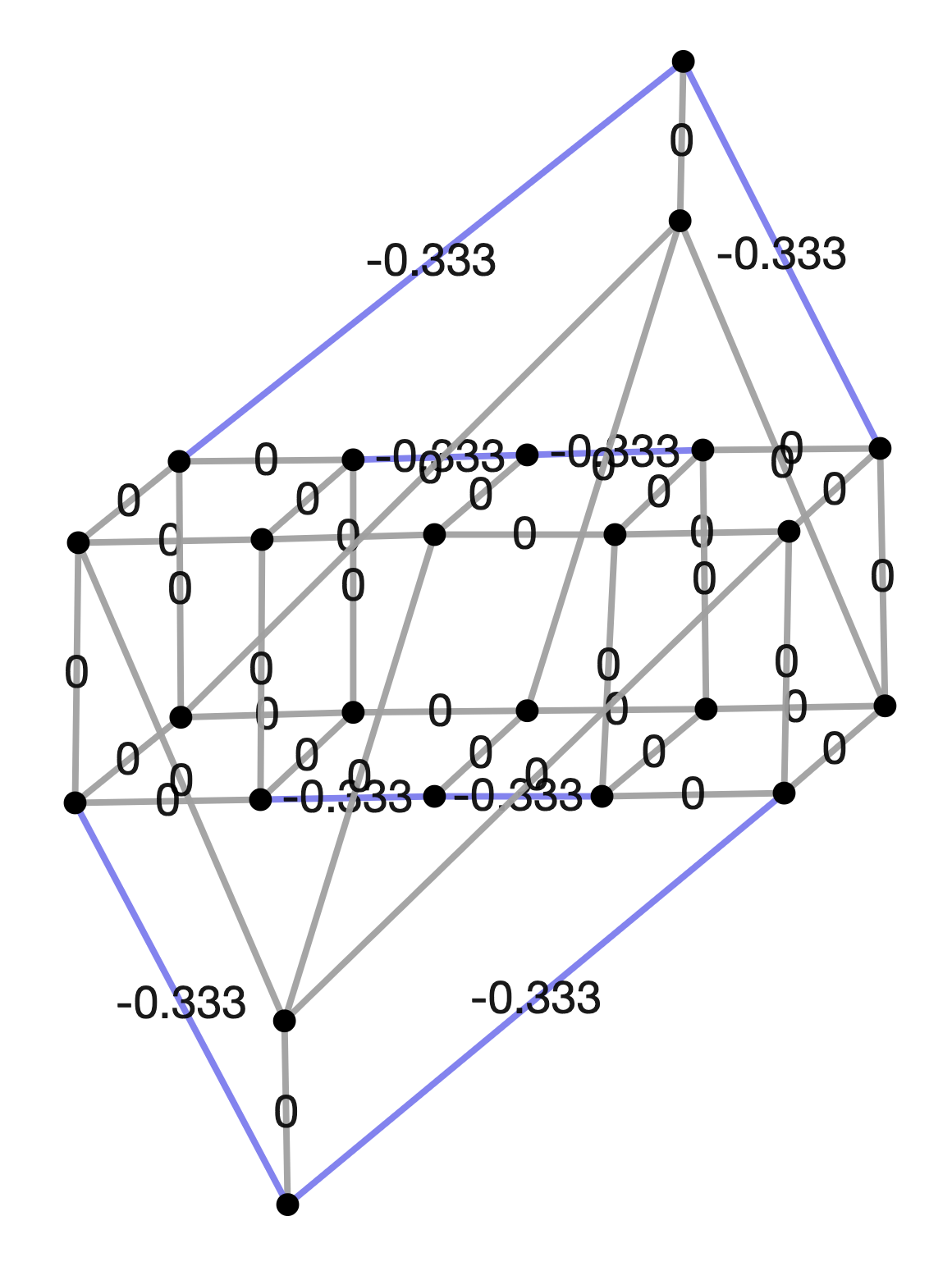}
\end{center}
\caption{The Handa graph has non-positive non-normalized Bakry-\'Emery curvature (left) and non-positive Ollivier Ricci curvature $\kappa_0$ (right).}
\label{fig:handa2}
\end{figure}
\end{ex}

The following remark mentions some similarities between Bonnet-Myers sharpness with respect to Steinerberger curvature and Bonnet-Myers sharpness with respect to Lin-Lu-Yau curvature $\kappa_{LLY}$. 

\begin{rmk} The Bonnet-Myers Theorem for graphs with non-negative Lin-Lu-Yau curvature has the same 
form, namely
\begin{equation} \label{eq:BM-sharp-Oll} 
{\rm{diam}}(G) \le \frac{2}{\min_{x \sim y} \kappa_{LLY}(x,y)}, 
\end{equation}
and we call a graph Ollivier Ricci Bonnet-Myers sharp, if \eqref{eq:BM-sharp-Oll} holds with equality. Regular Ollivier Ricci Bonnet-Myers sharp graphs  were considerd in \cite{CKKLMP-20} with the following results:
\begin{itemize}
\item[(a)] Regular Ollivier Ricci Bonnet-Myers sharpness is preserved under Cartesian product, provided a certain relation between their vertex degrees and their diameters is satisfied (Theorem 3.2 in \cite{CKKLMP-20});
\item[(b)] regular self-centered Ollivier Ricci Bonnet-Myers sharp have constant Lin-Lu-Yau curvature (Theorem 5.9 in \cite{CKKLMP-20});
\item[(c)] all regular self-centered Ollivier Ricci Bonnet-Myers sharp graphs are strongly spherical (Proposition 1.12 and Theorem 1.13).
\end{itemize}
We like to mention that strongly spherical graphs are graphs which are both antipodal and spherical (see \cite{KMS-04} and \cite{BHM-03} for the precise definitions), which means that the set of all regular self-centered Ollivier Ricci
Bonnet-Myers sharp graphs is a subset of all antipodal graphs and that they are therefore also all Bonnet-Myers sharp with respect to Steinerberger curvature, by Theorem \ref{thm:BM-equiv}. However, there are antipodal graphs which are not spherical, for example all cycles $C_{2n}$ of even length $2n \ge 6$. These graphs are Bonnet-Myers sharp with respect to Steinerberger curvature but not Ollivier Ricci Bonnet-Myers sharp (in fact, they have Lin-Lu-Yau curvature 0). Moreover, properties (a) and (c) of regular Ollivier Ricci Bonnet-Myers sharp graphs lead to a complete classification of regular self-centered Ollivier Ricci Bonnet-Myers sharp graphs (Theorem 1.6 in \cite{CKKLMP-20}). Note however that these results about Ollivier Ricci Bonnet-Myers sharp graphs are no longer valid if the regularity condition is dropped and there are many open problems about Ollivier Ricci Bonnet-Myers sharp graphs which are not self-centered and/or non-regular.
Kamtue \cite{K-20} showed that all regular Ollivier Ricci Bonnet-Myers sharp graphs of diameter $3$ are automatically self-centered. If one chooses non-normalized Lin-Lu-Yau curvature, M\"unch \cite{Mu-19} showed that regularity is no longer required for the classification of all self-centered Ollivier Ricci Bonnet-Myers sharp graphs. 
Recently, Cushing and Stone found some non-regular Ollivier Ricci Bonnet-Myers examples in the class of anti-trees (see \cite{CS-24}).
\end{rmk}

Next, we like to mention the following observation. (It was already mentioned in \cite{KMS-04} that any antipodal graph of diameter $2$ must be a Cocktail party graph.)

\begin{prop}
  The only Bonnet-Myers sharp graphs of diameter $2$ are the cocktail party graphs $CP(m)$ with $m \ge 2$. 
\end{prop}

\begin{proof}
    Bonnet-Myers sharpness of a graph $G=(V,E)$ of diameter $2$ implies, for any $x \in V$,
    $$ \sum_{y \in V} d(x,y) = n. $$
    Since $d(x,x) = 0$ and $d(x,y) \ge 1$ for all $y \neq x$, there must be exactly one vertex $y \in V$ with
    $d(x,y) = 2$ and all others need to be adjacent to $x$. Since this holds for every vertex $x \in X$, $G$ must be a complete graph $K_{2m}$ ($m \ge 2$) with a perfect matching removed and, therefore, the cocktail party graph $CP(m)$.   
\end{proof}

Let us finish this note with another straightforward relation between curvature and distance, involving the \emph{average distance function} of a connected graph $G=(V,E)$, given by
$$ d^\#(G) := \frac{1}{|V|^2} \sum_{x,y \in V} d(x,y). $$

\begin{prop} \label{prop:avdistcurv}
Let $G=(V,E)$ be a finite connected graph
admitting a solution $K: V \to \mathbb{R}_{\ge 0}$ for \eqref{eq:K-defin}. Then we have
\begin{equation}\label{eq:curvavdist} 
d^\#(G) \le \frac{1}{\min_x K(x)}. 
\end{equation}
Moreover, \eqref{eq:curvavdist} holds with equality if and only if the function $K$ is constant. 
\end{prop}

\begin{proof}
    It follows from \eqref{eq:K-defin} that
    $$ \min_z K(z) \sum_{x,y \in V} d(x,y) \le \sum_{x,y \in V} K(x) d(x,y) = n^2. $$
    This proves \eqref{eq:curvavdist}. This inequality holds with equality if and only if 
    $$ \min_z K(z) = K(x) \quad \text{for all $x \in V$}, $$ that is, $K$ is the constant function.
\end{proof}

\begin{rmk}
    A much more involved result similar to Proposition \ref{prop:avdistcurv} in the context of Ollivier Ricci curvature can be found in \cite[Theorem 1.1]{Mu-22b}. Namely, the following inequality holds for all finite connected graphs $G=(V,E)$:
    \begin{equation} \label{eq:avdistavOllcurv} 
    d^\#(G) \le \frac{{\rm{Deg}}^\#(G)}{\kappa^\#(G)}. 
    \end{equation}
    Here $\rm{Deg}: V \to \mathbb{N}$ is the combinatorial vertex degree and ${\rm{Deg}}^\#(G)$ its average, that is
    $$ 
    {\rm{Deg}}^\#(G) = \frac{1}{|V|} \sum_{x \in V} {\rm{Deg}}(x). 
    $$
    The \emph{average curvature} $\kappa^\#(G)$ is a weighted average of the \emph{non-normalized} Ollivier Ricci curvature for pairs of adjacent vertices $x \sim y$, that is
    $$ \kappa^\#(G) = \frac{\sum_{e \in E} g(e){{\rm{Ric}}(e)}}{\sum_{e \in E} g(e)}. $$
    Here the non-normalized Ollivier Ricci curvature ${\rm{Ric}}(e)$ of an edge $e=\{x,y\}$ is defined via Kantorovich Duality as follows (see \cite[Theorem 2.1]{MW-19}:
    $$ {\rm{Ric}}(e) = \inf_{\Vert \nabla f\Vert_\infty =1 \atop f(y)-f(x)=1} \Delta f(x) - \Delta f(y), $$
    where $\Vert \nabla f\Vert_\infty = \sup_{x \sim y} |f(y)-f(x)|$ and $\Delta f(x) = \sum_{y \sim x} (f(y)-f(x))$ is the \emph{non-normalized} Laplacian. The weight $g(e)$ is the \emph{betweenness centrality} of the edge $e$, given by
    $$ g(e) = \frac{1}{|V|^2} \sum_{u,v \in V} {\rm{geod}}_{e}(u,v), $$
    where ${\rm{geod}}_e(u,v)$ is the quotient of the number of geodesics from $u \in V$ to $v \in V$ containing the edge $e$ and the number of all geodesics from $u$ to $v$. Of course, we have for $e=\{x,y\}$ that ${\rm{geod}}_{e}(x,y) = 1$, since all our graphs are assumed to be simple, that is, without multiple edges and without loops. For more details, we refer the readers to \cite{Mu-22b}. Moreover, that paper gives also a full classification of the graphs satisfying \eqref{eq:avdistavOllcurv} with equality, namely that they are precisely the reflective graphs, which are Cartesian products of cocktail party graphs, Johnson graphs, halved cubes, Schl\"afli and Gosset graphs (see also the predecessor paper \cite{Mu-22a}).
\end{rmk}

{\bf{Acknowledgements:}} We would like to thank Bertie and Blossom for providing us with inspiring mathematical thoughts regarding Theorem \ref{thm:bertie-blossom}. Moreover, we thank Pakawut Jiradilok for useful discussions. Shiping Liu is supported by the National Key R and D Program of China 2020YFA0713100 and the National Natural Science Foundation of China No. 12031017. Supanat Kamtue is supported
by Shuimu Scholar Program of Tsinghua University No. 2022660219.

\end{document}